\documentclass[12pt,reqno]{amsart}

\usepackage[utf8]{inputenc}
\usepackage{enumerate}
\usepackage{mathtools}
\usepackage{amssymb}
\usepackage{comment}

\usepackage{hyperref}
\usepackage{geometry}

\newgeometry{vmargin={25mm}, hmargin={25mm,25mm}}

\usepackage[style=alphabetic, backend=biber, giveninits=true]{biblatex}
\theoremstyle{plain}
\newtheorem{thm}{Theorem}[section]
\newtheorem{prop}[thm]{Proposition}
\newtheorem{cor}[thm]{Corollary}
\newtheorem{lem}[thm]{Lemma}
\newtheorem{clm}[thm]{Claim}

\theoremstyle{remark}
\newtheorem*{rmk}{Remark}

\theoremstyle{definition}

\renewcommand{\a}{{\alpha}}

\newcommand{\g}{{\gamma}}

\newcommand{\G}{{\Gamma}}

\newcommand{\C}{{\mathbb C}}
\newcommand{\D}{{\mathbb D}}

\newcommand{\R}{{\mathbb R}}
\newcommand{\N}{{\mathbb N}}
\newcommand{\Z}{{\mathbb Z}}

\newcommand{\T}{{\mathbb T}}

\newcommand{\cA}{{\mathcal A}}
\newcommand{\cB}{{\mathcal B}}
\newcommand{\cC}{{\mathcal C}}
\newcommand{\cE}{{\mathcal E}}

\renewcommand{\cB}{{\mathcal B}}

\newcommand{\cM}{{\mathcal M}}

\newcommand{\cQ}{{\mathcal Q}}

\newcommand{\cD}{\mathcal{D}}

\newcommand{\dd}{\mathrm{d}}

\newcommand{\goldR}{\varphi}
\newcommand{\veps}{\varepsilon}
\newcommand{\vphi}{\varphi}
\newcommand{\greenf}{\mathfrak{g}}

\DeclareMathOperator{\Vol}{Vol}
\DeclareMathOperator{\LCap}{Cap}

\DeclareMathOperator{\Area}{Area}
\DeclareMathOperator{\Length}{Length}

\newcommand{\vol}[1]{\Vol\left(#1\right)}
\newcommand{\logcap}[1]{\LCap\left({#1}\right)}

\renewcommand\emptyset{\varnothing}

\newcommand{\disc}[2]{\D\left({#1},{#2}\right)}

\newcommand\defeq{\coloneqq}
\newcommand\eqdef{\eqqcolon}

\title[Integer-valued polynomials]{Integer-valued polynomials\\ satisfying growth constraints}
\author{Avner Kiro}
\author{Alon Nishry}
\date{}

\address{A. Kiro, Faculty of Mathematics and Computer Science,
The Weizmann Institute of Science,
234 Herzl Street,Rehovot 76100, Israel}
\email{avner-ephraiem.kiro@weizmann.ac.il}
\address{A. Nishry, School of Mathematical Sciences, Tel Aviv University, Tel Aviv 69978, Israel}
\email{alonish@tauex.tau.ac.il}

\thanks{The research of AN was supported in part by ISF Grant 3537/24.}


\begin{document}
\begin{abstract}
We consider polynomials which take integer values on the integers (IVPs), and satisfy an additional growth condition on the natural numbers. Elkies and Speyer, answering a question by Dimitrov, showed there is a critical exponential growth threshold, such that there are infinitely many IVPs with growth above the threshold and finitely many IVPs below that threshold (of arbitrary degree). In this paper, we give more refined estimates for the number of IVPs having exponential growth thresholds. In addition, we consider a similar problem, where there is a (not necessarily symmetric) growth condition on the integers. Notably, the critical threshold is determined by the logarithmic capacity of an explicit domain.
\end{abstract}

\maketitle

\section{Introduction}
Ruzsa's conjecture asks whether a function $f : \N \to \Z$ which satisfies
\[
f(m) - f(n) \equiv 0 \pmod{m - n} \quad \text{for all } n,m\in\N
\]
and $|f(n)|<e^n$ for all $n\in\N$ is a polynomial (see \cite{Zannier96}). Perhaps being motivated by this conjecture, Dimitrov \cite{Dim13} asked if there are infinitely many \emph{integer-valued polynomials} $P$ of arbitrary degree that satisfy
\begin{equation}
    \label{eq:P_N_growth_cond}
    |P(n)| \le A^n \mbox{ for all } n \in \N = \{ 0, 1, 2, \dots \},
\end{equation}
where $A > 1$ is fixed. Recall that an integer-valued polynomial (IVP) $P$ is an algebraic polynomial taking integer values on the natural numbers $\N$ (or equivalently on the integers $\Z$). For more on IVPs see the book \cite{CaCha97} and also the survey by Cahen and Chabert \cite{CaCha16}.

Elkies and Speyer \cite{Dim13} proved that the threshold lies at the \emph{golden ratio}, that is, if $A < \goldR = \frac12 (1 + \sqrt{5})$ there are only finitely many such polynomials (of all degrees), while if $A > \goldR$ there are infinitely many of them. To the best of our knowledge the answer for $A = \goldR$ is still not known.

Our first theorem provide a small (asymptotic) improvement of the existence result.

\begin{thm}\label{thm:IVP_thres_infty}
For any $\veps > 0$, there are infinitely many integer-valued polynomials $P$ that satisfy the inequalities
\[
|P(n)| \le \left(\sqrt{2\goldR} + \veps \right) \goldR^n, \quad \text{for all } n \in \N.
\]
\end{thm}

\begin{rmk}
We are not aware of any explicit example of an infinite family of IVPs satisfying the above inequalities (even after replacing the constant $\sqrt{2\goldR}$ by any larger absolute constant).
\end{rmk}

It is also possible to give $\ell^2$-type results.
\begin{thm}\label{thm:IVP_thres_elltwo}
There are no (non-trivial) IVPs $P$ satisfying
\[
\sum_{n\ge 0} |P(n)|^2 \goldR^{-2n} < \goldR.
\]
For any $\veps>0$, there are infinitely many IVPs $P$ that satisfy
\[
\sum_{n\ge 0} \tfrac{1}{n+1} |P(n)|^2 \goldR^{-2n} \le \frac{2}{\pi}  + \veps.
\]
\end{thm}

\begin{rmk}
Notice that for the constant IVP, $P\equiv1$, it holds
\[
\sum_{n\ge 0} \goldR^{-2n} = \goldR, 
\]
and moreover, taking $P(n) = \tfrac12 n(n+3)$, one can check that
\[
\sum_{n\ge 0} \frac{1}{n+1} |P(n)|^2 \goldR^{-2n} = 4 \goldR^2 \log (\goldR)-\frac{15}{2 \goldR} \approxeq 0.40 < \frac{2}{\pi}.
\]
This bound can be improved, but we do not know the optimal bound (if one exists).
\end{rmk}

It is known \cite[Chap. 1]{CaCha97} that the sequence of binomial polynomials $\{ \binom{x}{k} \}_{k\in\N}$ forms a basis over $\Z$ for IVPs. That is, an IVP $P$ of degree at most $d$ can be written uniquely in the form
\[
P(x) = P_\textbf{c}(x) = \sum_{k=0}^d c_k \binom{x}{k}, \quad \mbox{with} \quad \textbf{c} = (c_0, \dots, c_d) \in \Z^{d+1}.
\]
Hence, counting IVPs $P$ that satisfy the inequalities \eqref{eq:P_N_growth_cond} is equivalent to counting integer lattice points inside the \emph{convex} body
\[
\cC_A(d) = \left\{ \textbf{c} = (c_0, \dots, c_d) \in \R^{d+1} \colon \sup_{n\in\N} |P_\textbf{c}(n) A^{-n}| \le 1 \right\}.
\]

It is well known that in order to bound the number of $\Z^{d+1}$ lattice points inside $\cC_A(d)$ from below, it is sufficient to estimate $\vol{\cC_A(d)}$. On the other hand, to obtain an upper bound for the number of lattice points one has to use the special structure of this convex set. 
\begin{rmk}
Note that since a polynomial of degree $d$ is determined by its values at (say) $\{ 0, \dots, d \}$, there are finitely many polynomials of \emph{fixed} degree $d$ that satisfy any of the inequalities in the theorems above.
\end{rmk}

\subsection*{Acknowledgement}
We thank Aron Wennman for providing references for orthogonal polynomials, and Misha Sodin and Fedja Nazarov for helpful conversations.

\section{The results}
We are interested in estimating the number of IVPs of degree at most $d$ satisfying growth conditions on the integers such as \eqref{eq:P_N_growth_cond} as the degree $d$ tends to infinity.

We begin by introducing some notation. Let $A,B > 1$ be fixed constants. We put $\textbf{c} = \textbf{c}(d) = (c_0, \dots, c_d)$, and define the convex body
\[
\cC_{A,B}^\infty(d;t) = \left\{ \textbf{c} \in \R^{d+1} \colon \sup_{n\in\N} |P_\textbf{c}(n) A^{-n}| \le t , \: \sup_{n\in\N^+} |P_\textbf{c}(-n) B^{-n}| \le t \right\},
\]
where $\N = \{ 0, 1, 2, \dots \},\, \N^+ = \{ 1, 2, 3, \dots \}$. Notice the $\Z^{d+1}$ lattice points in $\cC_{A,B}^\infty(d;t)$ correspond to IVPs $P$ of degree at most $d$ satisfying the bounds
\begin{equation*}
    |P(n)| \le t A^n, \quad \forall n \in \N, \quad \text{and} \quad |P(-n)| \le t B^n, \quad \forall n \in \N^+.
\end{equation*}
Similarly, one may consider the $\ell^2$ bounds
\begin{equation*}
    \sum_{n\in\N}|P(n)|^2 A^{-2n} \le t^2 , \quad \text{and} \quad \sum_{n\in\N^+} |P(-n)|^2 B^{-2n} \le t^2,
\end{equation*}
corresponding to $\Z^{d+1}$ lattice points inside the set
\[
\cC_{A,B}^2(d;t) = \left\{ \textbf{c} \in \R^{d+1} \colon \sum_{n\in\N}|P_\textbf{c}(n)|^2 A^{-2n} \le t^2, \: \sum_{n\in\N^+} |P_\textbf{c}(-n)|^2 B^{-2n} \le t^2 \right\}.
\]

It turns out that the volume of the convex bodies above is closely related to the \emph{logarithmic capacity} of certain domains in the complex plane. The logarithmic capacity of a compact set $E \subset \C$ is given by (see \cite[Chap. 5]{Ransford95})
\[
\logcap{E} \defeq \exp\left( -\inf_\nu \int_{E\times E} \log \frac{1}{|z-w|} \, \dd \nu(z) \, \dd \nu(w) \right),
\]
where the infimum is taken over \emph{probability} measures $\nu$ supported on the set $E$.

Put
\[
\D_1 = \disc{\frac{1}{A^2-1}}{\frac{A}{A^2-1}},\:
\D_2 = \disc{\frac{B^2}{1-B^2}}{\frac{B}{B^2-1}},
\]
where $\D(z,r)$ is the disk centered at $z$ and of radius $r$. Notice that if $A, B > 1$, then $\D_1 \cap \D_2 = \emptyset$. See Figure \ref{fig:circles}.

\begin{figure}[b]
    \centering
    \includegraphics[width=0.95\linewidth]{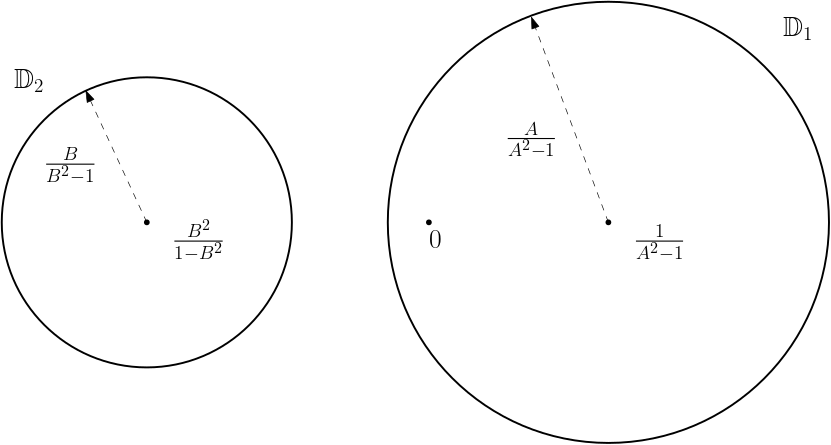}
    \caption{The disks $\mathbb{D}_1$ and $\mathbb{D}_2$.}
    \label{fig:circles}
\end{figure}

It is known that the capacity of a disk is equal to its radius (e.g. \cite[Cor. 5.2.2]{Ransford95}), e.g. $\logcap{\D_1} = \frac{A}{A^2-1}$. Note that the solution to the equation $\tfrac{A}{A^2-1} = 1$ with $A>1$ is the golden ratio $\goldR$. This value of $A$ is the \emph{critical threshold} for the volume of the set of polynomial coefficients ${\mathbf c}$, for which the polynomial $P_{\mathbf c}$ satisfies the one-sided bound \eqref{eq:P_N_growth_cond}.

For the two-sided bounds, the following result gives the volume of the set of polynomial coefficients, and the asymptotic number of IVPs (depending on the degree) above the critical threshold. 
\begin{thm}\label{thm:subcrit_volume_latticept}
Let $A,B > 1, t > 1$ be fixed constants, $\D_1, \D_2$ as above, and
\[
\g_{A,B} \defeq \logcap{\D_1 \cup \D_2}.
\]
Then, as $d \to \infty$,
\[
\begin{aligned}
\log \vol{\cC_{A,B}^2(d;t)} & = -\tfrac12 \log \g_{A,B} \cdot d^2 - \tfrac12 d \log d + O(d \log t), \quad \text{and}\\
\log \vol{\cC_{A,B}^\infty(d;t)} & = -\tfrac12 \log \g_{A,B} \cdot d^2 + O(d \log t).
\end{aligned}
\]
Moreover, when $\g_{A,B} \in (0,1)$ we have
\[
\log \# \left\{ \cC_{A,B}^2(d;t) \cap \Z^{d+1} \right\} = \log \vol{\cC_{A,B}^2(d;t)} + O(d \log t ).
\]
\end{thm}

\begin{rmk}
By the monotonicity of logarithmic capacity, we see that the region where $\gamma_{A,B} \in (0,1)$ is contained in the set $A, B \ge \goldR$.
\end{rmk}

\begin{rmk}
Using the Green function for an annulus (e.g. \cite[V.15.7]{CoHi53}) one can derive an expression for the capacity of the union of two disjoint disks in terms of theta functions. More precisely, one can show that
\[
\g_{A,B} = \tfrac12 \exp\left( \frac{\log^2 A + \log^2 B}{2 \log (AB)} \right) \frac{\theta_1^\prime(0,\a)}{\sqrt{|\theta_1(i \log A, \a) \theta_1(i \log B, \a)|}},
\]
where $\a = (AB)^{-1}$, $\theta_1$ is a theta function (see e.g. \cite[16.27 f.]{AbSt92}). Here we used a modification of a formula due to Ransford, see  \cite[Example 4.7]{LiSN17}.

See Figure \ref{fig:level_set_capacity} for an illustration of the level sets of $\gamma_{A,B}$ using Mathematica. The dashed line corresponds to the critical level set $\gamma_{A,B} = 1$.
\end{rmk}

\begin{figure}
    \centering
    \includegraphics[width=0.95\linewidth]{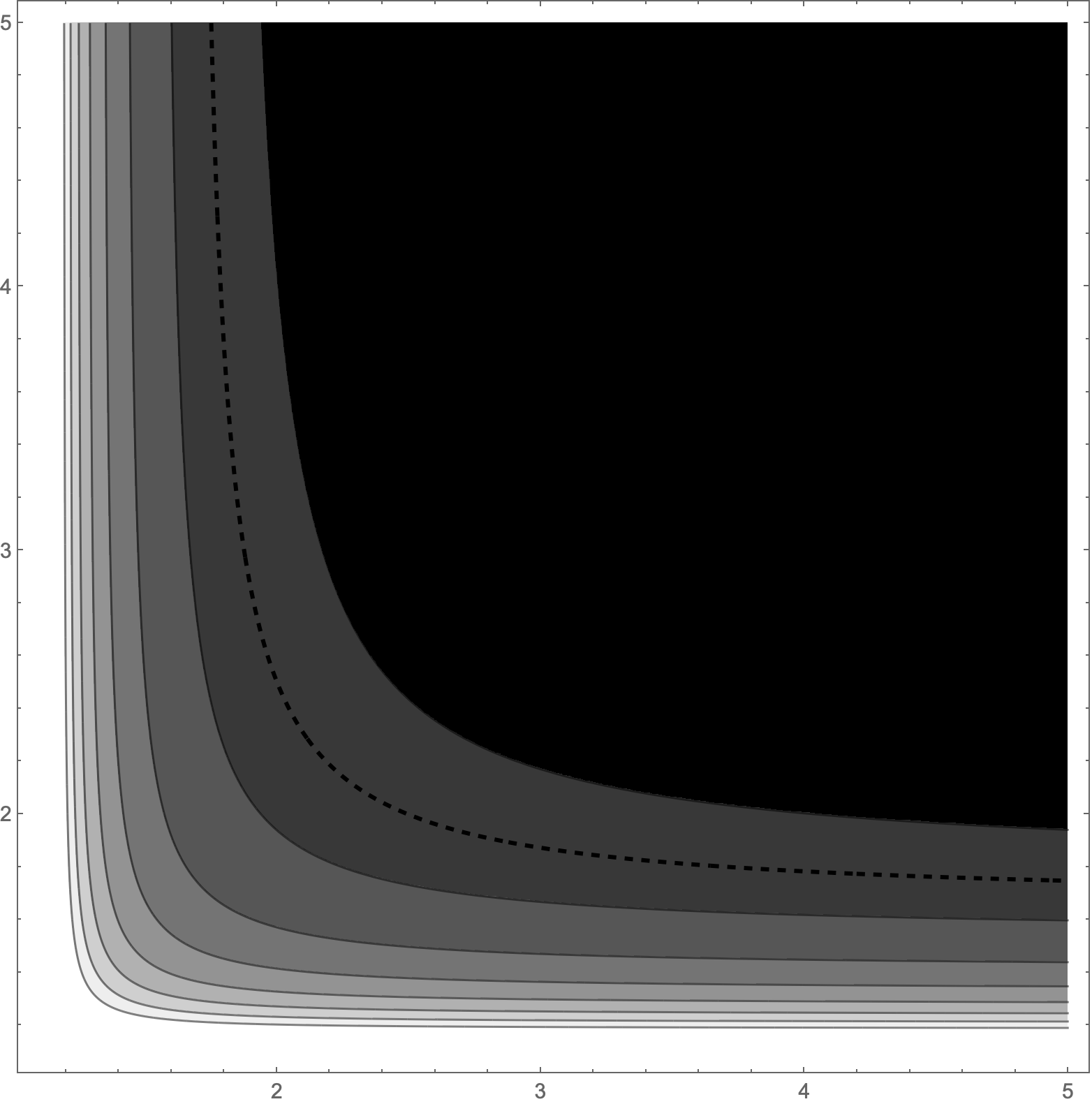}
    \caption{Level sets of $\gamma_{A,B}$}
    \label{fig:level_set_capacity}
\end{figure}

At the critical threshold, we can prove the analogs of Theorems~\ref{thm:IVP_thres_infty} and \ref{thm:IVP_thres_elltwo}, up to the calculation of the explicit constants (which depend on $A, B$).

\begin{thm}\label{thm:linfty_twosided}
Let $A,B > 1$ be fixed constants as in Theorem \ref{thm:subcrit_volume_latticept} that satisfy $\gamma_{A,B} = 1$. There is a constant $C_1 = C_1(A,B) > 0$ so that there are infinitely many IVPs $P$ satisying
\[
 |P_\textbf{c}(n)| \le C_1 A^n , \:\: \text{for all } n\in \N \text{ and } |P_\textbf{c}(-n)| \le C_1 B^n , \:\: \text{for all } n\in \N^+.
\]
\end{thm}

\begin{thm}\label{thm:ltwo_twosided}
Let $A,B > 1$ be fixed constants as in Theorem \ref{thm:subcrit_volume_latticept} that satisfy $\gamma_{A,B} = 1$. There is a constant $C_2 = C_2(A,B) > 0$ so that there are no IVPs $P$ that satisfy
\[
\sum_{n\in\N} |P_\textbf{c}(n)|^2 A^{-2n} \le C_2, \text{ and }  \:\: \sum_{n\in\N^+} |P_\textbf{c}(-n)|^2 B^{-2n} \le C_2.
\]
Moreover, there is a constant $C_3 = C_3(A,B) > 0$, such that there are infinitely many IVPs $P$ satsifying
\[
\sum_{n\in\N} \frac{|P_\textbf{c}(n)|^2}{n+1}  A^{-2n} \le C_3, \text{ and }  \:\: \sum_{n\in\N^+} \frac{|P_\textbf{c}(-n)|^2}{n+1} B^{-2n} \le C_3.
\]
\end{thm}

\subsection*{Notation}
We will use the following notation for the disk and the circle.
$\D(a,r) = \{ z \in \C \colon |z-a| < r \}$, $\T(a,r) = \partial \D(a,r)$. We also write $\D = \D(0,1)$, $\T = \partial \D$ for the unit disk and unit circle.

We write $\dd|z|$ to denote integration with respect to arc-length measure, and $\dd A(z)$ denotes integration with respect to area measure (both are not normalized). That is, if $\G$ is a (rectifiable) curve, and $G$ is a domain, then
\[
\int_\G 1 \, \dd |z| = \textrm{length}(\G), \qquad \int_G 1 \, \dd A(z) = \textrm{area}(G).
\]

\section{Outline of the proofs of Theorems \ref{thm:IVP_thres_infty} and \ref{thm:IVP_thres_elltwo}}
Here we give a summary of the ideas behind the proofs of the theorems with a growth condition on the natural numbers. Using the binomial polynomial basis
\[
\binom{z}{k} = \frac{z(z-1)\dots(z-k+1)}{k!}, \quad k \in \N,
\]
we let
\[
P(z) = P_{\bf c}(z) = \sum_{k=0}^d c_k \binom{z}{k}
\]
be a polynomial of degree $d$. The polynomial $P_{\bf c}$ is integer-valued if and only if ${\bf c} = (c_0, \dots, c_d) \in \Z^{d+1}$. The key observation is that the \emph{generating function}
\[
g_P(z) \defeq \sum_{n\in\N} P(n) z^n = \frac{1}{1-z} \sum_{k=0}^d c_k \left( \frac{z}{1-z} \right)^k, \qquad \forall |z| < 1,
\]
 can be utilized to obtain information on $\textbf{c}$ given the sequence $\{ P(n) \}_{n\in\N}$ and vice versa.

\subsection{The \texorpdfstring{$\ell^\infty$}{linf} constraint}
Let $t > 0$. We first consider the (easier) problem of the asymptotic volume of the set
\[
\left\{ \textbf{c} \in \R^{d+1} \colon \sup_{n\in\N} |P_\textbf{c}(n) A^{-n}| \le t \right\}.
\]
Using theorems by Ball and Vaaler, and some approximations (see Section~\ref{sec:prelim} for the details), we will arrive at the following result
\[
\frac{t^{d+1}}{\sqrt{\det \cM}} \le \vol{ \Big\{ \textbf{c} \in \R^{d+1} \colon \sup_{n\in \N} | P_{\bf c}(n) r^n | \le \tfrac12 t \Big\} } \le \frac{C_1^{d+1} t^{d+1}}{\sqrt{\det \cM}},
\]
where $C_1 > 0$ is an absolute constant, and
\[
\cM = \cM(d;r) = \left[ \sum_{n=0}^\infty \binom{n}{j} \binom{n}{k} r^{2n} \right]_{j,k=0}^{d}.
\]

Observe that by the orthogonality of the monomials
\[
\begin{aligned}
\sum_{n=0}^\infty P_1(n) P_2(n) r^{2n} & = \frac{1}{2 \pi r} \int_{r \T} \left[ \sum_{m=0}^\infty P_1(m) w^m \right] \left[ \sum_{n=0}^\infty P_2(n) \bar{w}^n \right] \dd |w| \\
& = \frac{1}{2 \pi r} \int_{r \T} g_{P_1}(w) \overline{g_{P_2}(w)} \, \dd |w|,
\end{aligned}
\]
where the integrals are with respect to arc-length.

By a change of variables, using the fractional linear transformation $w = z/(1-z)$, we conclude that
\[
\cM = \left[ \frac{1}{2 \pi r} \int_{\T_1} w^j \bar{w}^k \, \dd |w| \right]_{j,k=0}^{d}
\]
is a \emph{moment matrix}, where
\[
\D_1 = \D\left(\frac{r^2}{1-r^2},\frac{r}{1-r^2}\right), \quad
\T_1 = \partial \D_1 = \T\left(\frac{r^2}{1-r^2},\frac{r}{1-r^2}\right).
\]

It is known that (see e.g. Appendix \ref{sec:orthoPoly})
\[
\det \cM = \prod_{k=0}^d \| q_k \|^2_{L^2(\mu)},
\]
where $q_k$ is the \emph{monic} orthogonal polynomial (of degree $k$) with respect to $\mu$ -- the arc-length measure on $\T_1$ (weighted by $\frac{1}{2 \pi r}$). These polynomials and their norms are given explicitly by
\[
q_k(w) = \big( w - \tfrac{r^2}{1-r^2} \big)^k,
\qquad
\| q_k \|^2_{L^2(\mu)} = \frac1r \Big( \frac{r}{1-r^2} \Big)^{2k+1},
\]
hence
\[
\det \cM = \frac{1}{r^{d+1}} \Big( \frac{r}{1-r^2} \Big)^{(d+1)^2}.
\]
The critical threshold is achieved when $\frac{r}{1-r^2} = 1$, that is $r = \goldR^{-1}$. Theorem~\ref{thm:IVP_thres_infty} which concerns lattice points is obtained using a related theorem of Vaaler.

\subsection{The \texorpdfstring{$\ell^2$}{l2} constraint}
Here we have, for $r < 1$,
\[
\frac{1}{2 \pi r} \int_{r \T} |g_P(z)|^2 \dd |z| = \sum_{n=0}^\infty |P(n)|^2 r^{2n} = \| a_n \|^2_{\ell^2}, 
\]
and
\[
\frac{1}{\pi r^2} \int_{r \D} |g_P(z)|^2 \dd A(z) = \sum_{n=0}^\infty \frac{|P(n)|^2}{n+1} r^{2n} = \Big\| \frac{a_n}{\sqrt{n+1}} \Big\|^2_{\ell^2},
\]
where we write $a_n = a_n(r) = P(n) r^n$, and the second integral is with respect to planar Lebesgue measure.

Again making the change of variables $w = z/(1-z)$, we find that
\[
\| a_n \|^2_{\ell^2} = \frac{1}{2 \pi r} \int_{\T_1} |h_\textbf{c}(w)|^2 \dd |w|,
\qquad \Big\| \frac{a_n}{\sqrt{n+1}} \Big\|^2_{\ell^2} = 
\frac{1}{\pi r^2} \int_{\D_1} |h_\textbf{c}(w)|^2 \frac{\dd A(w)}{|1+w|^2},
\]
where $\D_1$ and $\T_1$ are as above, and
\[
h_{\bf c}(w) = \sum_{k=0}^d c_k w^k.
\]
Let $q_k$ be the (Szeg\H{o}) monic orthogonal polynomial with respect to the measure $\mu$ as above, and $t_k$ be the (Bergman) monic orthogonal polynomial, with respect to weighted area measure
\[
\dd \nu(w) = \frac{\dd A(w)}{\pi r^2 |1+w|^2}\quad\text{on}\quad \D_1.
\]
We are not aware of any simple formula for the $t_k$s (see e.g. \cite[Remark 8.1]{BergmanPoly}). However, by results of Korovkin (see Section \ref{subsec:volume_estimates}) we have, with $r = \goldR^{-1}$, and some $q \in (0,1)$,
\[
\| t_k \|^2_{L^2(\nu)} = \frac{1}{k+1} + O(q^k).
\]

We obtain the correspondences
\[
\sum_{n=0}^\infty |P_{\bf c}(n)|^2 \goldR^{-2n} \le t^2 \iff
\sum_{k=0}^d |\alpha_k({\bf c})|^2 \| q_k \|^2_{L^2(\mu)} =
\goldR \sum_{k=0}^d |\alpha_k({\bf c})|^2 \le t^2,
\]
where $h_{\bf c} = \sum_{k=0}^d \alpha_k({\bf c}) q_k$, and
\[
\sum_{n=0}^\infty \frac{|P(n)|^2}{n+1} \goldR^{-2n} \le t^2 \iff
\sum_{k=0}^d |\beta_k({\bf c})|^2 \| t_k \|^2_{L^2(\nu)} \le t^2,
\]
where now $h_{\bf c} = \sum_{k=0}^d \beta_k({\bf c}) t_k$. The proof of the first part of Theorem~\ref{thm:IVP_thres_elltwo} follows by observing that since $q_k$ are monic polynomials, if $h_{\bf c}$ has integer coefficients, then the first, in \emph{decreasing} order of $k$, non-zero coefficient $\alpha_k({\bf c})$ is the same integer. The second part of the theorem is deduced from a classical theorem of van der Corput on the existence of lattice points inside symmetric convex bodies with a large volume.

\section{Preliminaries}\label{sec:prelim}

\subsection{Generating functions}
To an algebraic polynomial $P$ we associate the generating function
\begin{equation}\label{eq:gen_func_gP}
g(z) = g_P(z) = \sum_{n=0}^\infty P(n) z^n,    
\end{equation}
which converges for $|z| < 1$. We also denote by $g_k$ the generating function of the binomial polynomial $\binom{\cdot}{k}$. We leave the proof of the following claim as an exercise to the reader.
\begin{clm}\label{clm:binom_generating_func}
For $|z| < 1$ and $k\ge 0$, we have
\[
g_k(z) = \frac{1}{1-z} \left( \frac{z}{1-z} \right)^k.
\]
\end{clm}
The above claim provides an analytic continuation of $g_P$ to a meromorphic function.
\begin{clm}
For any polynomial $P$, the function $g_P$ is analytic in $\C \setminus \{1\}$ with $g(\infty) = 0$, and in addition
\begin{equation}\label{eq:gen_func_gP_Laurent}
g_P(z) = -\sum_{n=1}^\infty \frac{P(-n)}{z^n} \quad \mbox{for all } |z| > 1.     
\end{equation}
\end{clm}
\begin{rmk}
For a more general claim see [Levin, Lecture $10$, Remark $1$].
\end{rmk}
\begin{proof}
Any polynomial $P$ of degree $d$ can be written in a unique way as
\[
P(x) = \sum_{k=0}^d c_k \binom{x}{k},
\]
hence it is sufficient to prove the claim for binomials.
By Claim \ref{clm:binom_generating_func}, for $k \in \N$ we have that
\[
g_k(z) = \sum_{n\in\N} \binom{n}{k} z^n = \frac{z^k}{(1-z)^{k+1}}, \qquad |z| < 1,
\]
where the right hand side is analytic in $z \in \C \setminus \{1\}$ and vanishes at infinity. Let $\veps \in (0,1)$. Using the vanishing of $g_k$ at infinity, and Cauchy's integral formula, we have
\[
\binom{n}{k} = \frac{1}{2\pi i} \int_{|w|=\veps} \frac{w^k}{(1-w)^{k+1}} \frac{\dd w}{w^{n+1}} = -\frac{1}{2\pi i} \int_{|w-1|=\veps} \frac{w^k}{(1-w)^{k+1}} \frac{\dd w}{w^{n+1}}.
\]

We will use $w^z = \exp(z \log w)$, where the principal branch of $\log$, cut along the ray $(-\infty, 0]$ is chosen. Observe that with this choice,
\[
Q(z) \defeq -\frac{1}{2\pi i} \int_{|w-1|=\veps} \frac{w^k}{(1-w)^{k+1}} \frac{\dd w}{w^{z+1}},  
\]
is an entire function with at most polynomial growth at infinity (this can be seen by choosing $\veps$ of order $|z|^{-1}$). Hence by Liouville's theorem $Q$ \emph{is} a polynomial. Since the polynomials $Q$ and $\binom{\cdot}{k}$ coincide on the (infinite) set $\N$, we conclude they are equal. In particular, for $n \in \N$,
\[
\binom{-n}{k} = -\frac{1}{2\pi i} \int_{|w-1|=\veps} \frac{w^{k+n}}{(1-w)^{k+1}} \frac{\dd w}{w}.
\]
By making a change of variables $\zeta = 1/w$, so that $\dd w / w = -\dd \zeta / \zeta$, and shifting the contour we find
\[
\binom{-n}{k} = \frac{1}{2\pi i} \int_{|\zeta^{-1}-1|=\veps} \frac{1}{(\zeta-1)^{k+1}} \frac{\dd \zeta}{\zeta^n} = \frac{1}{2\pi i} \int_{|\zeta|=\veps} \left[-\frac{1}{(\zeta-1)^{k+1}} \right] \frac{\dd \zeta}{\zeta^n}.
\]
Therefore, again using Cauchy's formula, for sufficiently small $|\zeta|$,
\[
\sum_{n\ge 1} \binom{-n}{k} \zeta^{n-1} = -\frac{1}{(\zeta - 1)^{k+1}} \implies
\sum_{n\ge 1} \binom{-n}{k} \zeta^{n} = -\frac{\zeta}{(\zeta - 1)^{k+1}}.
\]
Finally, putting $z = 1/\zeta$ we obtain
\[
\sum_{n\ge 1} \binom{-n}{k} z^{-n} = -\frac{z^k}{(1-z)^{k+1}} = -g_k(z), \quad \text{ where } |z| > 1,
\]
as claimed.
\end{proof}

The following properties of the generating function are crucial for us.

\begin{cor}\label{cor:integral_GF}
For a polynomial $P(z) = \sum_{k=0}^d c_k \binom{z}{k}$ of degree $d$, the function $g_P$ can be analytically continued to $\C \setminus \{1\}$ so that
\[
g_P(z) = \frac{1}{1-z} \sum_{n=0}^d c_k \left( \frac{z}{1-z} \right)^k, \quad z \in \C \setminus \{1\}.
\]
Moreover, for $r > 0$,
\[
\frac{1}{2 \pi r} \int_{r \T} |g_P(w)|^2 \dd |w| = \begin{cases}
\sum_{n=0}^\infty |P(n)|^2 r^{2n}, & r < 1;\\
\sum_{n=1}^\infty |P(-n)|^2 r^{-2n}, & r > 1,
\end{cases}
\]
and
\begin{align*}
\frac{1}{\pi r^2} \int_{r \D} |g_P(w)|^2 \dd A(w) & = \sum_{n=0}^\infty \frac{|P(n)|^2}{n+1} r^{2n},\quad \text{when } r < 1,\\
\frac{r^2}{\pi} \int_{(r \D)^c} |g_P(w)|^2 \, \frac{\dd A(w)}{|w|^4} & =  \sum_{n=1}^\infty \frac{|P(-n)|^2}{n+1} r^{-2n},\quad \text{when } r > 1.
\end{align*}
\end{cor}
\begin{rmk}
Under the conformal map $z \mapsto \tfrac1z$ the area measure $\dd A(z)$ transforms to $\frac{\dd A(z)}{|z|^4}$.
\end{rmk}
\begin{proof}
The first part follows from the previous claims. The second part follows by expanding $|g_P(w)|^2$ as in \eqref{eq:gen_func_gP} and \eqref{eq:gen_func_gP_Laurent} in the variables $w, \overline{w}$ and observing that
\[
\frac{1}{2 \pi r} \int_{r \T} w^j \overline{w}^k \dd |w| = r^{2 j} \delta_{j,k}, \quad \text{when } j,k \in \Z.
\]
Similarly, the formulas with respect to area measure follow from
\[
\frac{1}{\pi} \int_{r \D} w^j \overline{w}^k \dd A(w) = \frac{r^{2(j+1)}}{j+1} \delta_{j,k}, \quad \text{when } r<1 \text{ and } j,k \in \N,
\]
and
\[
\frac{1}{\pi} \int_{(r \D)^c} \frac{1}{w^{j}} \frac{1}{\overline{w}^{k}} \, \frac{\dd A(w)}{|w|^4}  = \frac{r^{-2(j+1)}}{j+1} \delta_{j,k}, \quad \text{when } r>1 \text{ and } j,k \in \N.
\]
\end{proof}

The next result follows from the above by polarization. We leave the details to the reader.
\begin{cor}\label{cor:innerProd_GF}
For $j, k\ge 0$ we have,
\[
\sum_{n=0}^\infty \binom{n}{j} \binom{n}{k} r^{2n} = \frac{1}{2\pi r} \int_{r \T} g_j(w) \overline{g_k(w)} \dd |w|, \quad r \in (0,1),
\]
and
\[
\sum_{n=1}^\infty \binom{-n}{j} \binom{-n}{k} r^{-2n} = \frac{1}{2\pi r} \int_{r \T} g_j(w) \overline{g_k(w)} \dd |w|, \qquad r > 1.
\]
\end{cor}

\subsection{Tools for the \texorpdfstring{$\ell^\infty$}{linf} case}
Let $N, k\in \N$ with $N>k$, and let $\Sigma$ be an $N \times k$ matrix of maximal rank $k$. We use the following version of a result by Vaaler \cite[Thms. 1 and 2]{Vaa79}, see also Gluskin-Kashin \cite{glu92}.

\begin{thm}\label{thm:Vaaler_volume_lattice_pts}
We have that
\[
V_N \defeq \vol{ \Big\{ \textbf{c} = (c_1, \dots, c_k) \in \R^{k} \colon \sup_{n\le N} \Big|\sum_{j=1}^k c_j \Sigma_{n,j} \Big| \le \tfrac12 \Big\} } \ge \frac{1}{\sqrt{\det \Sigma^t \Sigma}}.
\]
Moreover, if $R \in \N$ is such that
\[
R^2 \cdot \det \Sigma^t \Sigma \le 1,
\]
then there are at least $2R$ non-trivial integer solutions $\textbf{c} \in \Z^k$ to the system of inequalities
\[
\left| \sum_{j=1}^k c_j \Sigma_{n,j} \right| \le 1 \qquad \mbox{for all} \quad n\in \{1, \dots, N\}.
\]
\end{thm}

In the other direction we use a result of Ball \cite{Ball89} (also Gluskin-Kashin \cite{glu92}).
\begin{thm}\label{thm:Ball_volume}
We have that
\[
\vol{ \Big\{ \textbf{c} = (c_1, \dots, c_k) \in \R^{k} \colon \sup_{n\le N} \Big|\sum_{j=1}^k c_j \Sigma_{n,j}\Big| \le 1 \Big\} } \le \frac{\sqrt{2}^{N - k}}{\sqrt{\det \Sigma^t \Sigma}}.
\]
\end{thm}

For the convenience of the reader we prove the following standard result in Appendix \ref{sec:orthoPoly}.
\begin{thm}\label{thm:det_moment_matrix}
Let $\mu$ be a measure on $\C$ with infinite support and finite moments of all orders. Define the $n$-th moment matrix by
\[
\cM(n;\mu) = \left[ M_{k,m} \right]_{k,m=0}^n, \qquad M_{k,m} = \int_\C z^j \bar{z}^k \, \dd \mu(z).
\]
Then
\[
\det \cM(n;\mu) = \prod_{j=0}^n \left\| Q_j \right\|^2_{L^2(\mu)},
\]
where $Q_j$ are the orthogonal monic polynomials, with respect to $\mu$.
\end{thm}

Let $\eta \in (0,1)$. Suppose $\G = \T_1 \cup \T_2$ is the union of two circles with \emph{disjoint} interiors, whose radii are bounded between $\eta$ and $\frac{1}{\eta}$, and are separated by at least $\eta$ and at most $\frac{1}{\eta}$. Denote the $n \times n$ moment matrix, with respect to arc-length measure on $\G$, by $\cM_{\G}(n)$. We will need an effective lower bound for the smallest eigenvalue of $\cM_{\G}(n)$, which is given by
\[
\lambda_{n}^{-1} = \max \left\{ \frac{1}{2\pi} \int_{-\pi}^\pi \left| P(e^{i\theta}) \right|^2 \dd \theta \, \colon \, \deg P = n, \,\, \int_\G |P(z)|^2 |\dd z| = 1  \right\}.
\]
The following result is obtained by combining \cite[Theorem 7.3 and Eq. (10.5)]{Widom69}.
\begin{prop}\label{prop:bound_least_eigenvalue}
Using the same notation as above, and under the same assumptions, we have
\[
\lambda_{n} \ge \exp(-C_{\G} n),
\]
where the constant $C_{\G} > 0$ depends continuously on the Green function of the exterior of $\G$ (and its derivatives).
\end{prop}

\begin{rmk}
Under the assumptions above the constant $C_{\G}$ can be bounded in terms of $\eta$ only.
\end{rmk}

\subsubsection{Elementary results}

We use the standard bound (that follows from the inequality $n! \ge (n/e)^n$)
\begin{equation}\label{eq:bound_binomial}
\binom{n}{k} \le \left( \frac{e n}{k} \right)^k, \qquad n > k \ge 1.    
\end{equation}

\begin{clm}\label{clm:simple_bound}
Let $a > 0$ and $C \ge e$, then we have
\[
\sup_{1\le x \le a/C} \left( \frac{a}{x} \right)^x \le C^{a/C}.
\]
\end{clm}
\begin{proof}
The function $x \log \tfrac{a}{x}$ is increasing on $[1,a/C]$.
\end{proof}

\begin{lem}\label{lem:poly_bound}
For every $A > 1$ there exists a constant $C_A > 1$ such that if $P$ is a polynomial of degree $d$ satisfying $|P(n)| \le A^n$ for all $0 \le n \le C_A d$, then $|P(n)| \le A^n$ for all $n \in \N$.
\end{lem}

\begin{proof}
The Lagrange interpolation formula states that
\[
P(x) = \sum_{j=0}^d P(j) L_j(x),\qquad \mbox{where} \quad L_j(x) = \prod_{\overset{0\le m\le d}{m\ne j}} \frac{x-m}{j-m}.
\]
Notice that for $n \ge d+1$
\[
|L_j(n)| \le n^{d+1} \prod_{\overset{0\le m\le d}{m\ne j}} (j-m)^{-1} \le \frac{n^{d+1}}{[(\lfloor d/2\rfloor)!]^2}
\le \left( \frac{6 n}{d} \right)^{d+1}.
\]
Let $C_A$ be a constant depending only on $A$ to be chosen later. Under the assumption of the lemma, we have that $|P(n)|\le A^n$ for all $0 \le n \le C_A d$. In addition, for $n \ge d+1$
\[
|P(n)| \le \left( \frac{6 n}{d} \right)^{d+1} \sum_{j=0}^d A^j \le \frac{1}{A-1} \left( \frac{6 A n}{d} \right)^{d+1}.
\]
If we choose $C_A$ large enough so that
\[
\mathrm{(i)}\,\, \exp(2/C_A) \le A \quad \mathrm{(ii)}\,\, 6 A C_A \le A^{C_A/2} \quad \mathrm{(iii)}\,\, (A-1)^{-1} \le A^{C_A/2},
\]
then we have that
\[
\frac{1}{A-1} \left( \frac{6 A n}{d} \right)^{d+1} \le A^n, \qquad \mbox{for all}\quad n \ge C_A d.
\]
In particular, we may assume $C_A$ is decreasing with $A$.
\end{proof}

A similar result holds for $P(-n)$ (e.g. define $Q(z) = B^{-1} P(-z -1)$ and apply the lemma). Hence, one can prove the following generalization. We leave the details to the reader.

\begin{lem}\label{lem:poly_bound_twosided}
For every $A, B > 1$ there exists a constant $C_{A,B} > 1$ such that if $P$ is a polynomial of degree $d$ satisfying $|P(n)| \le A^n$ and $|P(-n)| \le B^n$ for all $0 \le n \le C_{A,B} \cdot d$, then $|P(n)| \le A^n$ for all $n \in \N$ and $|P(-n)| \le B^n$ for all $n \in \N$.
\end{lem}

The following lemma will allow us to approximate the value of the determinant of the moment matrix.
\begin{lem}\label{lem:bnd_entries_moment_mat}
Let $A > 1$, then for $L > C_A > e$ (where $C_A$ is the constant from Lemma~\ref{lem:poly_bound}), $n \in \N$ and $k,m\in \{0, \dots, n-1\}$ we have
\[
\sum_{j > L n} A^{-2j} \binom{j}{k} \binom{j}{m} \le \frac{A}{A-1} \exp((2 - L \log A) n).
\]
\end{lem}
\begin{proof}
We will prove the result for $k,m \ge 1$, the other cases being similar. For $j > Ln$, we have by \eqref{eq:bound_binomial} and Claim \ref{clm:simple_bound},
\[
A^{-2j} \binom{j}{k} \binom{j}{m}
\le e^{k + m} A^{-2j} \frac{j^{k+m}}{k^k m^m}
\le e^{2n} A^{-2j} L^{2j / L}.
\]
Therefore, if $L > e$ is large enough so that $L^{1/L} < \sqrt{A}$ we get
\[
\sum_{j > L n} A^{-2j} \binom{j}{k} \binom{j}{m} \le e^{2n} \sum_{j > L n} A^{-j},
\]
and the result follows.
\end{proof}

Similarly one can prove the following extension, we skip the details of the proof (one has to use $\left|\binom{-n}{k}\right| \le \binom{2n}{k}$ for $n\ge 0$ and $k \in \{0, \dots, n-1\}$).
\begin{lem}\label{lem:bnd_entries_moment_mat_twosided}
Let $A, B > 1$, then for $L > C_{A,B} > e$ (where $C_{A,B}$ is the constant from Lemma~\ref{lem:poly_bound_twosided}), $n \in \N$ and $k,m\in \{0, \dots, n-1\}$ we have
\[
\begin{aligned}
    \sum_{j < -L n} & B^{2j} \binom{j}{k} \binom{j}{m} + \sum_{j > L n} A^{-2j} \binom{j}{k} \binom{j}{m} \\
    & \le C^\prime_{A,B} \exp\left((4 - L \log \min\{A,B\}) n\right),
\end{aligned}
\]
where $C^\prime_{A,B}$, is some constant depending continuously on $A, B$.
\end{lem}

\subsubsection{Bounds for the determinant of the moment matrix}
Let $\cA$ be a $n \times n$ Hermitian matrix, and denote its eigenvalues by $\lambda_n(\cA) \le \dots \le \lambda_1(\cA)$. We will use the following result (see e.g. \cite[III.2]{Bhatia97}).
\begin{thm}[Weyl's inequalities] Let $\cA, \cB$ be $n \times n$ Hermitian matrices.
For all $k \in \{1, \dots, n\}$ we have
\[ \lambda_k(\cA) + \lambda_n(\cB) \le \lambda_{k}(\cA+\cB) \le \lambda_k(\cA) + \lambda_1(\cB). \]
\end{thm}

We will use the following notation for the entries of moment matrices:
\[
\begin{aligned}
\Xi_1(A,k,m; M) & = \sum_{j=0}^M A^{-2j} \binom{j}{k} \binom{j}{m}, \\
\Xi_2(A,B,k,m; M) & = \sum_{j=-M}^{-1} B^{2j} \binom{j}{k} \binom{j}{m} + \sum_{j=0}^M A^{-2j} \binom{j}{k} \binom{j}{m}.
\end{aligned}
\]

\begin{prop}\label{prop:det_lower_bnd}
Let $d \in \N$, $\veps \in (0,1/2)$, and $A \in \left(1 + \veps, 1/\veps \right)$. Consider the following (symmetric) matrices
\[
\begin{aligned}
\cM^{(1)}_d(A) & = \left[ \Xi_1(A,k,m,\infty) \right]_{k,m=0}^{d},\\
\cM^{(1)}_d(A;M) & = \left[ \Xi_1(A,k,m,M) \right]_{k,m=0}^{d}.
\end{aligned}
\]
If $L > e$ is sufficiently large (depending only on $\veps$), then
\[
\det \cM^{(1)}_d(A; \lfloor Ld \rfloor) \ge 2^{-(d+1)} \det \cM^{(1)}_d(A).
\]
\end{prop}
\begin{rmk}
Notice that $\cM^{(1)}_d(A)$ and $\cM^{(1)}_d(A;M)$ are positive-definite matrices. By the Cauchy-Binet formula we always have that
\[
\det \cM_d^{(1)}(A; M) \le \det \cM_d^{(1)}(A).
\]
\end{rmk}
\begin{proof}
We let $\cA = \cM_d^{(1)}(A)$ and define the matrix
\[
 \cB = \left[ \sum_{j > \lfloor Ld \rfloor}^\infty A^{-2j} \binom{j}{k} \binom{j}{m} \right]_{k,m=0}^{d}.
\]
By Proposition \ref{prop:bound_least_eigenvalue} there is a constant $C_1 = C_1(\veps)$ such that $\lambda_{d+1}(\cA) \ge \exp(-C_1 d)$. Lemma~\ref{lem:bnd_entries_moment_mat} gives a pointwise bound for the entries of $\cB$, which in turn bounds its operator norm (up to factor $d+1$ which can be subsumed into the exponential). Thus we may choose $L = L(\veps)$ sufficiently large so that for all $d$ large enough $\lambda_1(\cB) \le \tfrac12 \lambda_{d+1}(\cA)$. 

Hence, Weyl's inequalities give
\begin{align*}
\det(\cM_d(A; \lfloor Ld \rfloor)) & = \det(\cA - \cB) = \prod_{j=1}^{d+1} \lambda_j(\cA - \cB) \\
& \ge \prod_{j=1}^{d+1} [ \lambda_j(\cA) - \lambda_1(\cB) ]
 \ge \prod_{j=1}^{d+1} \tfrac12 \lambda_j(\cA) = 2^{-(d+1)} \det(\cA).
\end{align*}
Note that the constant $\tfrac12$ is arbitrary and can be replaced by (e.g.) $1 - \exp(-d)$.    
\end{proof}

Using Lemma \ref{lem:bnd_entries_moment_mat_twosided}, one can prove the following generalization.
\begin{prop}\label{prop:det_lower_bnd_twosided}
Let $d \in \N$, $\veps \in (0,1/2)$, and $A, B \in \left(1 + \veps, 1/\veps \right)$. Consider the following (symmetric) matrices
\[
\begin{aligned}
\cM^{(2)}_d(A,B) & = \left[ \Xi_2(A,B,k,m,\infty) \right]_{k,m=0}^{d},\\
\cM^{(2)}_d(A,B;M) & = \left[ \Xi_2(A,B,k,m,M) \right]_{k,m=0}^{d}.
\end{aligned}
\]
If $L > e$ is sufficiently large (depending only on $\veps$), then
\[
\det \cM^{(2)}_d(A,B; \lfloor Ld \rfloor) \ge 2^{-(d+1)} \det \cM^{(2)}_d(A,B).
\]
\end{prop}

\subsection{Tools for the \texorpdfstring{$\ell^2$}{l2} case}
Below we will need an estimate for the number of lattice points inside ellipsoids of a special type.
\begin{prop}\label{prop:lattice_pts_in_ellipse}
Let $d \in \N$ and $\gamma, t > 0$ be parameters (that may depend on $d$), and define the ellipsoid
\[
\cE = \cE(d ; \gamma, t) = \left\{ \textbf{x} = (x_0, \dots, x_d) \in \R^{d+1} \colon \sum_{k=0}^d x_k^2 \gamma^{2k} \le t^2 \right\}.
\]
Moreover, let $\Psi$ be a lower unipotent matrix, that is, an $n\times n$ lower triangular matrix, with diagonal elements all equal to $1$. We denote by $\Psi(\cE)$ the image of $\cE$ under $\Psi$. Note that this map preserves volume.

For $\gamma \le 1 \le t$ we have
\[
-C_1 (d+1) \le \log \frac{\# \left( \Psi(\cE) \cap \Z^{d+1} \right)}{\vol{\cE}} \le  C_2 (d+1) \log \min\{ d+1, (1-\gamma)^{-1} \},
\]
where $C_1, C_2>0$ are absolute constants.
\end{prop}
\begin{rmk}
Notice that $\vol{\cE} = \gamma^{-d(d+1)/2} t^{d+1} \vol{B_2^{d+1}}$.
\end{rmk}
\begin{proof}
The lower bound (with $C_1 = \log 2$) follows from \cite[Sec. 7, Thm. 4]{GeomOfNumbers}.

Let $\mathbf{z}_1, \mathbf{z}_2 \in \Psi(\cE) \cap \Z^{d+1}$ be two different lattice points. Notice that by the structure of the matrix $\Psi$, we have that
\[
\| \mathbf{x}_1 - \mathbf{x}_2 \|_\infty \defeq \| \Psi^{-1}(\mathbf{z}_1) - \Psi^{-1}(\mathbf{z}_2) \|_\infty \ge 1.
\]

Now let $\mathbf{x} = (x_0, \dots, x_d) \in \cE$ be one of the preimages of a $\Z^{d+1}$ lattice point in $\Psi(\cE)$. If $\mathbf{y} = (y_0, \dots, y_d) \in \left[-\tfrac12, \tfrac12\right]^{d+1}$, then
\begin{align*}
\sum_{k=0}^d (x_k + y_k)^2 \gamma^{2k} & = \sum_{k=0}^d x_k^2 \gamma^{2k} + 2 \sum_{k=0}^d x_k y_k \gamma^{2k} + \sum_{k=0}^d y_k^2 \gamma^{2k}\\
& \le 2\left[ t^2 + \tfrac14 \min\{d+1, (1-\gamma^2)^{-1}\} \right],
\end{align*}
since $2|x_k y_k| \le x_k^2 + y_k^2$. Therefore, if we denote by $\cQ$ the union of all the cubes with side length $1$ centered at preimages of lattice points inside $\cE$, we see that
\[
\cQ \subset C_2 (1 + t^{-2} \min\{d+1, (1-\gamma)^{-1}\} \cE \subset C_2 \min\{d+1, (1-\gamma)^{-1}\} \cE.
\]
Since the cubes forming $\cQ$ are of volume $1$ and have disjoint interiors we get the required bound for the number of lattice points.
\end{proof}

The following theorem can be proved explicitly using conformal maps, but it also follows as a very special case of \cite[Thm. 9.1]{Widom69}. The fine properties of the normalizing constants can be found in \cite[Sec. 6]{Widom69}.
\begin{thm}\label{thm:Szego_OP}
Let $\eta \in (0,1)$. Suppose $\G = \T_1 \cup \T_2$ is the union of two circles with \emph{disjoint} interiors, whose radii are bounded between $\eta$ and $\frac{1}{\eta}$, and are separated by at least $\eta$ and at most $\frac{1}{\eta}$. Denote by $\left\{q_k\right\}_{k=0}^\infty$ the monic orthogonal polynomial of degree $k$ on $\G$ with respect to arc-length, and let
\[
\beta_k^2 = \int_\G |q_k|^2 \,\dd \sigma,
\]
be its normalization constant (squared). The asymptotics of $\beta_k$ are given by
\begin{equation}\label{eq:beta_k_asymp}
\beta_k = \logcap{\Gamma}^k f_k(\G) (1+o(1)), \qquad k\to\infty,
\end{equation}
where $f_k$ are functions that can be bounded between positive constants depending only on $\eta$. The implicit error term can also be bounded in terms of $\eta$ only.
\end{thm}

We will also require corresponding results for Bergman polynomials (i.e. orthogonal with respect to area measure). The following result closely follows \cite[Theorem 4.1]{BergmanPoly}. We give the details in Appendix~\ref{sec:Bergman_poly_two_islands}.
\begin{thm}\label{thm:BegmanAsymp}
Let $\eta \in (0,1)$. Suppose $G = \D_1 \cup \D_2$ is the \emph{disjoint} union of two disks, whose radii are bounded between $\eta$ and $\frac{1}{\eta}$, and are separated by at least $\eta$ and at most $\frac{1}{\eta}$.

Let $\{ t_k \}_{k=0}^\infty$ denote the sequence of monic Bergman orthogonal polynomials associated with $G$, that is
\[
\int_G t_{k_1} (z) t_{k_2}(z) \, \dd A(z) = \gamma_{k_1}^2  \cdot
\begin{cases}
1,    & k_1 = k_2, \\
0,    & k_1 \ne k_2,
\end{cases}
\]
where $\dd A$ denotes the Lebesgue measure on $\C$. There are constants $0 < C_1(\eta) < C_2(\eta)$ such that for $k\in \N$
\[
C_1 \frac{\logcap{\G}^{k}}{\sqrt{k+1}} \le \gamma_k \le C_2 \frac{\logcap{\G}^{k}}{\sqrt{k+1}}.
\]
\end{thm}

\section{\texorpdfstring{$\ell^2$}{l2} volume and number of lattice points}
In this section we prove Theorem \ref{thm:IVP_thres_elltwo}, part of Theorem \ref{thm:subcrit_volume_latticept} and Theorem \ref{thm:ltwo_twosided}.

We recall some notation. Put $\textbf{c} = \textbf{c}(d) = (c_0, \dots, c_d) \in \R^{d+1}$ and let
\begin{align*}
P(x) & = P_\textbf{c}(x) = \sum_{k=0}^d c_k \binom{x}{k},\\
g_\textbf{c}(z) & = g_P(z) = \frac{1}{1-z} \sum_{n=0}^d c_k \left( \frac{z}{1-z} \right)^k, \quad z \in \C \setminus \{1\}.    
\end{align*}

Let $A,B > 1$ be fixed constants, and define the sequences
\[
a_n(\textbf{c}) = P_\textbf{c}(n) A^{-n}, \: n\in\N, \quad b_n(\textbf{c}) = P_\textbf{c}(-n) B^{-n}, \: n\in\N^+,
\]
and the sets
\[
\cC_{d,A,B}^2(t) \defeq \left\{ \textbf{c} \in \R^{d+1} \colon \| a_n(\textbf{c}) \|_{\ell^2(\N)} \le t, \, \| b_n(\textbf{c}) \|_{\ell^2(\N^+)} \le t \right\},
\]
and
\[
\cD_{d,A,B}^2(t) \defeq \left\{ \textbf{c} \in \R^{d+1} \colon \left\| \frac{a_n(\textbf{c})}{\sqrt{n+1}} \right\|_{\ell^2(\N)} \le t, \, \left\| \frac{b_n(\textbf{c})}{\sqrt{n+1}} \right\|_{\ell^2(\N^+)} \le t \right\}.
\]

\subsection{\texorpdfstring{$\ell^2$}{l2} volume estimates}\label{subsec:volume_estimates}
We now estimate the volume of the sets $\cC^2$ and $\cD^2$.
By Corollary~\ref{cor:integral_GF} we have
\[
\cC_{d,A,B}^2(t) =
\left\{ \textbf{c} \in \R^{d+1} \colon 
\frac{A}{2\pi} \int_{A^{-1} \T} |g_\textbf{c}(z)|^2 \dd |z| \le t^2, \, \right.
 \left.  \frac{1}{2\pi B} \int_{B \T} |g_\textbf{c}(z)|^2 \dd |z| \le t^2 \: \right\},
\]
and
\[
\cD_{d,A,B}^2(t) = \left\{ \textbf{c} \in \R^{d+1} \colon
\frac{A^2}{\pi} \int_{A^{-1} \D} |g_\textbf{c}(z)|^2 \dd A(z) \le t^2,\, \right.
\left. \frac{B^2}{\pi} \int_{(B \D)^c}
|g_\textbf{c}(z)|^2 \, \frac{\dd A(z)}{|z|^4} \le t^2 \: \right\}.
\]

Notice that under a (one to one) change of variables $z = \phi(w)$, the arc-length and area integrals transform as follows,
\[
\int_{\phi(\G)} f(z) \, \dd |z| = \int_\G f(\phi(w)) |\phi^\prime(w)| \, \dd |w|, \qquad \int_{\phi(G)} f(z) \, \dd A(z) = \int_G f(\phi(w)) |\phi^\prime(w)|^2 \, \dd A(w).
\]

We now make the change of variables via the fractional linear transformation,
\[
w = \psi(z) = \frac{z}{1-z}, \quad \phi = \psi^{-1}, \quad
z = \phi(w) = \frac{w}{1+w}, \quad \phi^\prime(w) = \frac{1}{(1+w)^2},
\]
such that
\[
h_{\textbf{c}}(w) \defeq (1+w)^{-1} (g_{\textbf{c}} \circ \phi)(w) = \sum_{k=0}^d c_k w^k,
\]
and in particular $h_{\textbf{c}}$ is a polynomial. Moreover,
\[
\psi(A^{-1} \T) = \T\left(\frac1{A^2-1},\frac{A}{A^2-1}\right) \eqdef \T_1, \quad
\psi(B \T) = \T\left(\frac{B^2}{1-B^2},\frac{B}{B^2-1}\right) \eqdef \T_2.    
\]
Thus,
\[
\cC_{d,A,B}^2(t) =
\left\{ \textbf{c} \in \R^{d+1} \colon
\frac{A}{2\pi} \int_{\T_1} |h_\textbf{c}(w)|^2 \dd |w| \le t^2, \right.
\left. \frac{1}{2\pi B} \int_{\T_2} |h_\textbf{c}(w)|^2 \dd |w| \le t^2
\right\},    
\]
and
\[
\cD_{d,A,B}^2(t) = \left\{ \textbf{c} \in \R^{d+1} \colon
\frac{A^2}{\pi} \int_{\D_1} |h_\textbf{c}(w)|^2 \frac{\dd A(w)}{|1+w|^{2}} \le t^2, \right.
\left. \frac{B^2}{\pi} \int_{\D_2} |h_\textbf{c}(w)|^2 \frac{\dd A(w)}{|w|^{4} |1+w|^{-2}} \le t^2
\right\},
\]
where
\[
\D_1 \defeq \D\left(\frac1{A^2-1},\frac{A}{A^2-1}\right), \quad
\D_2 \defeq \D\left(\frac{B^2}{1-B^2},\frac{B}{B^2-1}\right).
\]

\subsubsection{Volume of $\cC^2$}

Let $\{ q_k(w) \}_{k=0}^\infty$ be the sequence of monic orthogonal polynomials with respect to arc-length measure on the curve $\Gamma~\defeq~\T_1~\cup~\T_2$ in increasing degree. Then, we can write
\[
h_\mathbf{c}(w) = \sum_{k=0}^d \alpha_k(\mathbf{c}) q_k(w), \quad \text{where} \quad  \alpha_d = c_d, \qquad \beta_k^2 \defeq \int_\Gamma |q_k(w)|^2 \dd |w|.
\]

Furthermore,
\begin{equation}\label{eq:ellipse_comparison}
\cE(\beta_0,\dots,\beta_d; t_1) \subset \cC_{d,A,B}^2(t) \subset \cE(\beta_0,\dots,\beta_d; t_2),
\end{equation}
where $\cE(\beta_0,\dots,\beta_d;t)$ is the ellipse
\[
\cE(\beta_0,\dots,\beta_d;t) = \left\{ \textbf{c} \in \R^{d+1} \colon \sum_{k=0}^d |\alpha_k(\textbf{c})|^2 \beta_k^2 \le t^2 \right\},
\]
and
\[
t_1 \defeq \sqrt{2 \pi \min\{A^{-1}, B\}} \, t, \qquad t_2 \defeq \sqrt{2 \pi (A^{-1} +B)} \, t.
\]

Let $\eta \in (0,1)$, and suppose $A,B > 1$ satisfy the conditions of Theorem \ref{thm:Szego_OP}. In particular, $A$ and $B$ are bounded away from $1$ and $\infty$ in terms of $\eta$. Hence, the constants $\{ \beta_k \}_{k\in\N}$ satisfy
\[
\beta_k = \logcap{\Gamma}^k g_k(A,B,d),
\]
where $g_k$ are functions that are bounded between positive constants depending only on $\eta$. Therefore, by \eqref{eq:ellipse_comparison},
\[
\left[ C_1(\eta) t_1 \right]^{d+1} \le \frac{\vol{\cC_{d,A,B}^2(t)}}{\vol{B^{d+1}_2} \logcap{\Gamma}^{-d(d+1)/2}} \le \left[ C_2(\eta) t_2 \right]^{d+1}.
\]
This proves the $\ell^2$ volume asymptotics in the first part of Theorem \ref{thm:subcrit_volume_latticept}. The second part of the theorem follows by an application of Proposition \ref{prop:lattice_pts_in_ellipse}, with the ellipsoid\[
\Psi^{-1}\left(\cE(\beta_0,\dots,\beta_d; t_\ell)\right),
\]
where $\Psi$ is the change of basis given by the $\alpha_k$s.

Now, notice the leading coefficient of $h_\mathbf{c}$ is a non-zero integer. Hence, it follows that there is a constant $C_\eta > 0$ such that
\[
\int_{\T_1} |h_\textbf{c}(w)|^2 \dd |w| \ge C_\eta \quad \text{or} \quad
\int_{\T_2} |h_\textbf{c}(w)|^2 \dd |w| \ge C_\eta.
\]
If we choose $t$ sufficiently small, depending on $\eta$, then we conclude that there are no integer lattice points inside $\cC_{d,A,B}^2(t)$. This proves the first part of Theorem \ref{thm:ltwo_twosided}.

\subsubsection{Volume of $\cD^2$}
For Bergman orthogonal polynomials notice that first we can bound the weights
$|1+w|^{-2}$ on $\D_1$, and $|w|^{-4}|1+w|^2$ on $\D_2$ from above and from below in terms of $\eta$. Thus, there are constants $c^{(j)} = c^{(j)}(\eta), \,j\in\{1,2\},$ such that
\[
\begin{aligned}
\cD_{d,A,B}^2(t) & \subset \left\{ \textbf{c} \in \R^{d+1} \colon
c^{(1)} \int_{\D_1} |h_\textbf{c}(w)|^2 \dd A(w) \le t^2, \quad c^{(2)} \int_{\D_2} |h_\textbf{c}(w)|^2 \dd A(w) \le t^2
\right\}\\
& \subset \left\{ \textbf{c} \in \R^{d+1} \colon
\tfrac12 \min\{c^{(1)},c^{(2)}\} \left( \int_{\D_1} + \int_{\D_2} \right) |h_\textbf{c}(w)|^2 \dd A(w) \le t^2
\right\},
\end{aligned}
\]
and similarly in the other direction with a different constant.

Now, let $t_k$ be the Bergman monic orthogonal polynomials with respect to area measure on $\D_1 \cup \D_2$, and let $\gamma_k$ be their normalizing constants. By Theorem \ref{thm:BegmanAsymp} we have
\[
C_1(\eta) \frac{\logcap{\Gamma}^{k+1}}{\sqrt{k+1}} \le \gamma_k \le C_2(\eta) \frac{\logcap{\Gamma}^{k+1}}{\sqrt{k+1}},
\]
which, together with the previous discussion, gives
\[
\left( \widetilde{C_1}(\eta) t \right)^{d+1} \le \frac{\vol{\cD_{d,A,B}^2(t)}}{\logcap{\Gamma}^{-d(d+1)/2}} \le \left( \widetilde{C_2}(\eta) t \right)^{d+1}.
\]
Here we used the asymptotics of the volume of the Euclidean ball
\[
\vol{B^{d+1}_2} = \left(\frac{2\pi e (1+o(1))}{d+1}\right)^{(d+1)/2}, \quad \text{as} \quad d \to \infty.
\]
Hence, choosing $t$  sufficiently large, the volume of $\cD^2$ is exponentially large in $d$ (as long as $\logcap{\Gamma} \le 1$). Applying \cite[Sec. 7.2, Thm. 1]{GeomOfNumbers} proves the second part of Theorem \ref{thm:ltwo_twosided}.

\subsection{Refined estimate for the critical growth constraint on the natural numbers}
By essentially the same argument as in Section \ref{subsec:volume_estimates} we find
\begin{align*}
\cC_{d,A}^2(t) & = \left\{ \textbf{c} \in \R^{d+1} \colon
\sum_{n\in\N} |P_{\textbf{c}}(n) A^{-n}|^2 \le t^2 \right\}
 = \left\{ \textbf{c} \in \R^{d+1} \colon
\frac{A}{2\pi} \int_{\T_1} |h_\textbf{c}(w)|^2 \dd |w| \le t^2 \right\} \\ 
& = \left\{ \textbf{c} \in \R^{d+1} \colon
\frac{A}{2\pi} \sum_{k=0}^d |\alpha_k(\textbf{c})|^2 \beta_k^2 \le t^2 \right\},
\end{align*}
where as before
\[
\T_1 = \T\left(\frac1{A^2-1},\frac{A}{A^2-1}\right), \qquad h_{\textbf{c}}(w) = \sum_{k=0}^d c_k w^k.
\]
Here the monic orthogonal polynomials with respect to arc-length on $\T_1$, and their normalization constants are given explicitly by
\[
p_k(w) = (w - (A^2-1)^{-1})^k,\qquad \beta_k^2 = 2 \pi \left( \frac{A}{A^2-1} \right)^{2k+1}.
\]

Thus, taking the prefactor $A/(2\pi)$ into account, we have
\[
\vol{\cC_{d,A}^2(t)} =
\left( \frac{A^2-1}{A} \right)^{(d+1)^2/2} \vol{B^{d+1}_2} \left( \frac{t}{\sqrt{A}}\right)^{d+1}
 = \frac{(A^2-1)^{(d+1)^2/2}}{A^{(d+1)(d+2)/2}} \vol{B^{d+1}_2} t^{d+1}.
\]

Evidently if $c_k$ is the largest non-zero coefficient of $h_{\textbf{c}}$, then $\int_{\T_1} |h_\textbf{c}(w)|^2 \dd |w| \ge c_k^2 \beta_k^2$. Hence, when $A = \vphi$ and $t < \sqrt{\vphi}$ there are no integer-valued polynomials $P$, satisfying
\[
\sum_{n\in\N} |P_{\textbf{c}}(n) \varphi^{-n}|^2 \le t^2.
\]
This proves the first part of Theorem \ref{thm:IVP_thres_elltwo}. We now turn to the proof of the second part.

We will use the results from \cite[Chap. III]{Sue74}. We denote by  $\mathrm{S}.x$ and $\mathrm{S}.x.y$ theorem and equation numbers from this paper.

Denote
\begin{align*}
\cD_{d,\varphi}^2(t) & = \left\{ \textbf{c} \in \R^{d+1} \colon
\sum_{n\in\N} \frac{1}{n+1} |P_{\textbf{c}}(n) \varphi^{-n}|^2 \le t^2 \right\} \\
& = \left\{ \textbf{c} \in \R^{d+1} \colon
\frac{\varphi^2}{\pi} \int_{\D_1} |h_\textbf{c}(w)|^2 \frac{\dd A(w)}{|1+w|^2} \le t^2 \right\}\\
& = \left\{ \textbf{c} \in \R^{d+1} \colon
\frac{\varphi^2}{\pi} \sum_{k=0}^d |\alpha_k(\textbf{c})|^2 \gamma_k^2 \le t^2 \right\},
\end{align*}
where now $\gamma_k$ are the normalizing constants of monic orthogonal polynomials with respect to the weight $\frac{\dd A(w)}{|1+w|^2}$ on $\D_1 = \D(\varphi^{-1} , 1)$.

For fixed $\veps \in (0,1)$, we will find a condition on $t = t(\veps)$ so that the volume of $\cD_{d,\varphi}^2(t)$ is at least $(2+\veps)^{d+1}$. Then an application of \cite[Sec. 7.2, Thm. 1]{GeomOfNumbers} proves there are infinitely many integer-valued polynomials $P$ satisfying the growth condition
\[
\sum_{n\in\N} \frac{1}{n+1} |P_{\textbf{c}}(n) \varphi^{-n}|^2 \le t^2.
\]

The asymptotics of $\gamma_k$ follow from Theorem $\mathrm{S}.3.1$ (in Suetin's notation it is equal to $\lambda_k^{-1}$). We introduce some notation required for the statement of the results.

Let $\Phi$ be the conformal map from the exterior of $\D_1$ to the exterior of the unit disk $\D$, satisfying $\Phi(\infty) = \infty$ and $\Phi^\prime(\infty) > 0$, that is the map $\Phi(z) = z - \varphi^{-1}$.

We need to introduce an analytic function $D: \D^c \to \D_1^c$ such that
\[
\left|D(w)\right|^2 = \frac{1}{\left|1+\left(w + \varphi^{-1} \right)\right|^2} = \frac{1}{\left|w + \varphi\right|^2} \quad \text{on} \quad |w| = 1.
\]
Explicitly it is given by the following expression
\[
D(w) = \exp \left( -\frac{1}{4 \pi} \int_0^{2\pi} \log \frac{1}{| e^{i \theta} + \varphi|^2} \frac{e^{i\theta} + w}{e^{i \theta} - w} \, d\theta \right).
\]
Using Suetin's notation, let
\[
g(z) = \frac{\Phi^\prime(z)}{D\left(\Phi(z)\right)} = \frac{\Phi^\prime(z)}{D\left( z - \varphi^{-1} \right)} = \frac{1}{D(z - \varphi^{-1})}, \quad \text{and} \quad \alpha_0 = g(\infty) = \frac{1}{D(\infty)}.
\]
Since the function $z + \varphi$ does not vanish in $\D$, the mean value theorem for harmonic functions gives
\[
\log D(\infty) = \frac{1}{4 \pi} \int_0^{2\pi} \log \frac{1}{| e^{i \theta} + \varphi|^2} \, d\theta
 = - \frac{1}{2 \pi} \int_0^{2\pi} \log | e^{i \theta} + \varphi| \, d\theta
 = - \log \varphi.
\]

Now by Theorem $\mathrm{S}.3.1$ we have
\[
\gamma_k^2 = \frac{\pi}{\alpha_0^2(k+1)} \left(1 + o(1)\right) = \frac{\pi}{\varphi^2(k+1)} \left(1 + o(1)\right).
\]

By the asymptotic formula for the volume of Euclidean ball and Stirling's approximation, we have
\[
\vol{\cD_{d,\varphi}^2(t)} = \vol{B_2^{d+1}} \prod_{k=0}^{d} \gamma_k^{-1} \left(\frac{\sqrt{\pi t}}{\varphi} \right)^{d+1}
= (\sqrt{2\pi})^{d+1} t^{d+1} \cdot \exp(o(1)).
\]
Finally, we conclude that if $t \ge \sqrt{2/\pi} + \veps$, then the required condition is satisfied.

\section{\texorpdfstring{$\ell^\infty$}{linf} volume and lattice points}
Let $A,B > 1$ be fixed constants, and recall the sequences
\[
a_n(\mathbf{c}) = P_\mathbf{c}(n) A^{-n}, \: n\in\N, \quad b_n(\mathbf{c}) = P_\mathbf{c}(-n) B^{-n}, \: n\in\N^+.
\]
In this section we consider the sets
\[
\cC_{d,A,B}^\infty(t) = \left\{ \mathbf{c} \in \R^{d+1} \colon \| a_n(\mathbf{c}) \|_{\ell^\infty(\N)} \le t, \, \| b_n(\mathbf{c}) \|_{\ell^\infty(\N^+)} \le t \right\},
\]
and
\[
\cC_{d,A,B}^\infty(t;M) = \left\{ \mathbf{c} \in \R^{d+1} \colon \max_{n\in{\{0,\dots,M}\}}| a_n(\mathbf{c}) | \le t, \, \max_{n\in{\{1,\dots,M\}}}| b_n(\mathbf{c}) | \le t \right\}.
\]

By Lemma \ref{lem:poly_bound_twosided} there is a constant $C_{A,B} > 1$, such that for $M > C_{A,B} \cdot d$ we have
\[
\cC_{d,A,B}^\infty(t) = \cC_{d,A,B}^\infty(t;M) \quad \text{as sets}.
\]
Below we fix $M = \lceil C_{A,B} \cdot d \rceil$.

Let $\Sigma_M = \Sigma_M(d,t;A,B)$ be an $(2M+1) \times (d+1)$ matrix, whose entries are given by
\[
\Sigma_{j,k} = \begin{cases}
  \frac{1}{t} \binom{j}{k} B^{j}  & ,\, j \in \{-M,\dots,-1\}; \\
  \frac{1}{t} \binom{j}{k} A^{-j},  & j \in \{0, \dots, M\},
\end{cases}
\]
where $k\in\{0,\dots,d\}$. The semi-infinite matrix $\Sigma_\infty = \Sigma_\infty(d,t;A,B)$ is defined analogously. We have that
\[
\cC_{d,A,B}^\infty(t;M) = \left\{ (c_0,\dots,c_d) \in \R^{d+1} \colon \sup_{|j|\le M} \left| \sum_{k=0}^d c_k \Sigma_{j,k} \right| \le 1 \right\}.
\]

Denote by $V_M = V_M(d,t;A,B)$ the volume of the set $\cC_{d,A,B}^\infty(t;M)$ in $\R^{d+1}$. By Theorems~\ref{thm:Vaaler_volume_lattice_pts} and \ref{thm:Ball_volume} we have
\begin{equation}\label{eq:volume_upper_lower_bound}
\frac{1}{\sqrt{\det \Sigma_M^t \Sigma_M}} \le V_M \le \frac{2^{M-d/2}}{\sqrt{\det \Sigma_M^t \Sigma_M}} \le \frac{(C_{A,B})^{d}}{\sqrt{\det \Sigma_M^t \Sigma_M}}.    
\end{equation}
The Cauchy-Binet theorem gives
\[
\det \left(\Sigma_M^t \Sigma_M \right) \ge \det \left(\Sigma_\infty^t \Sigma_\infty \right).
\]
Notice that by Corollary \ref{cor:innerProd_GF} and a change of variables
\[
\begin{aligned}
\left( \Sigma_\infty^t \Sigma_\infty \right)_{j,k} & = \frac{1}{t^2} \left[ \sum_{n=-\infty}^{-1} \binom{n}{j} \binom{n}{k} B^{2n} + \sum_{n=0}^{\infty} \binom{n}{j} \binom{n}{k} A^{-2n} \right]\\
 & = \frac{1}{t^2} \left[ \sum_{n=1}^{\infty} \binom{-n}{j} \binom{-n}{k} B^{-2n} + \sum_{n=0}^{\infty} \binom{n}{j} \binom{n}{k} A^{-2n} \right]\\
& = \frac{1}{t^2} \left[ \frac{1}{\pi B} \int_{B\T} g_j(z) \overline{g_k(z)} \dd |z| + \frac{A}{\pi} \int_{A^{-1} \T} g_j(z) \overline{g_k(z)} \dd |z| \right]\\
& = \frac{1}{t^2} \left[ \frac{1}{\pi B} \int_{\T_2} w^j \bar{w}^k \dd |w| + \frac{A}{\pi} \int_{\T_1} w^j \bar{w}^k \dd |w| \right],
\end{aligned}
\]
where
\[
g_j(z) = \frac{1}{1-z} \left( \frac{z}{1-z} \right)^j, \quad \T_1 = \T\left(\frac1{A^2-1},\frac{A}{A^2-1}\right), \quad
\T_2 = \T\left(\frac{B^2}{1-B^2},\frac{B}{B^2-1}\right).
\]

By Theorem \ref{thm:det_moment_matrix} we have that
\[
\det \left( \Sigma_\infty^t \Sigma_\infty \right) = \prod_{k=0}^d \left\| q_k \right\|_{L^2(\mu_t)}^2,
\]
where $q_k$ are monic orthogonal polynomials with respect to arc-length on $\T_1 \cup \T_2$, weighted by $A$ on $\T_1$ and $\frac{1}{B}$ on $\T_2$. The measure $\mu_t$ is defined similarly, scaled by $\frac{1}{\pi t^2}$.

Let $\eta \in (0,1)$. Assume $A,B > 1$ are such that the condition of Theorem \ref{thm:Szego_OP} is satisfied. Then, we have that
\[
\left\| q_k \right\|_{L^2(\mu_t)} \asymp_\eta t^{-1} \logcap{\Gamma}^k, \quad \text{where} \quad \Gamma = \T_1 \cup \T_2.
\]
Notice we can reduce to the case of arc-length (weighted by $\frac{1}{t}$), by bounding the weight of the measure $\mu_t$ from below and from above, by factors depending only on $\eta$ (ignoring the dependence on $A,B$). Hence,
\begin{equation}\label{eq:volume_Sigma_infty_asymp}
\begin{aligned}
\log \det \left( \Sigma_\infty^t \Sigma_\infty \right) & = d(d+1) \log \logcap{\Gamma} - 2 (d+1) \log t + O_\eta( d ) \\
& = d^2 \log \logcap{\Gamma} - 2 d \log t + O_\eta( d ).    
\end{aligned}
\end{equation}
Therefore,
\[
\log \vol{\cC_{d,A,B}^\infty(t)} \ge -\frac12 d^2 \log \logcap{\Gamma} + d \log t + O_\eta( d ).
\]
By Proposition \ref{prop:det_lower_bnd_twosided} (perhaps after increasing the constant $C_{A,B}$ in the definition of $M$), and the upper bound in \eqref{eq:volume_upper_lower_bound} we obtain a matching upper bound for $\log \vol{\cC_{d,A,B}^\infty(t)}$.
This proves the $\ell^\infty$ volume result in the first part of Theorem \ref{thm:subcrit_volume_latticept}. It remains to prove Theorems~\ref{thm:IVP_thres_infty} and~\ref{thm:linfty_twosided}.

Fix $A,B > 1$ such that $\gamma_{A,B} = \logcap{\Gamma} = 1$ (in particular, we may assume $A,B > \frac32$). By the second part of Theorem \ref{thm:Vaaler_volume_lattice_pts}, if $\det \left( \Sigma_\infty^t \Sigma_\infty \right)$ (which depends on $d$) tends to $0$ as $d \to \infty$, then there are infinitely many IVPs satisfying the condition of Theorem \ref{thm:linfty_twosided}. Hence by \eqref{eq:volume_Sigma_infty_asymp} for $t$ sufficiently large depending on $\eta$ the required result is obtained.

To prove Theorem \ref{thm:IVP_thres_infty} we follow a similar path, now taking the constants into account. Consider the set
\[
\cC_{d}^\infty(t) = \left\{ \mathbf{c} \in \R^{d+1} \colon \sup_{n\in\N} |P_\mathbf{c}(n) \goldR^{-n}| \le t \right\}
= \left\{ \mathbf{c} = (c_0,\dots,c_d) \colon \sup_{n \in \N} \left| \sum_{k=0}^d c_k \Sigma_{n,k} \right| \le \frac12 \right\},
\]
where
\[
\Sigma_{n,k} = \Sigma_{n,k}(t) = \frac{1}{2t} \binom{n}{k} \goldR^{-n}.
\]
We have that
\[
\left( \Sigma_\infty^t \Sigma_\infty \right)_{j,k} = \frac{1}{4t^2} \sum_{n=0}^{\infty} \binom{n}{j} \binom{n}{k} \goldR^{-2n}
 = \frac{1}{4t^2} \frac{\goldR}{\pi} \int_{\goldR^{-1} \T} g_j(z) \overline{g_k(z)} \dd |z|
 = \frac{1}{4t^2} \frac{\goldR}{\pi} \int_{\T_1} w^j \bar{w}^k \dd |w|,
\]
where
\(
\T_1 = \T\left(1/\goldR, 1\right).
\)

Let $q_k$ be the monic orthogonal polynomial with respect to the arc-length measure on $\T_1$, which are given by
\[
q_k(w) = \big( w - 1/\goldR \big)^k.
\]
Hence, by Theorem \ref{thm:det_moment_matrix} we have (taking the normalization into account)
\[
\det \left( \Sigma_\infty^t \Sigma_\infty \right) = \left( \frac{\goldR}{2t^2} \right)^{d+1}.
\]

By Lemma \ref{lem:poly_bound}, there exists some constant $C_0 > 0$, such that we have 
\[
\cC_{d}^\infty(t) = \left\{ \mathbf{c} = (c_0,\dots,c_d) \colon \sup_{n \le C_0 d} \left| \sum_{k=0}^d c_k \Sigma_{n,k} \right| \le \frac12 \right\},
\]
for all $d$. Therefore, by the second part of Theorem \ref{thm:Vaaler_volume_lattice_pts} (and the monotonicity of the determinant), the set $\cC_{d}^\infty(2t)$ contains at least
\(
2 \left\lfloor \det \left( \Sigma_\infty^t \Sigma_\infty \right)^{-1/2} \right\rfloor
\)
integer lattice points. Hence for any $\veps > 0$, if we choose $t = \sqrt{\goldR/2} + \veps/2$ there are infinitely many IVPs $P$ satisfying
\[
|P(n)| \le \left(\sqrt{2\goldR} + \veps \right) \goldR^n \quad \text{for all } \, n \ge 0.
\]
This concludes the proof of Theorem \ref{thm:IVP_thres_infty}.

\appendix

\section{Moment matrix and orthogonal polynomials} \label{sec:orthoPoly}

Let $\mu$ be a measure on $\C$ with infinite support and all moments (so that no non-trivial algebraic polynomial can vanish on the support of $\mu$). Define the $n$-th moment matrix
\[
\cM(n;\mu) = \left[ M_{k,m} \right]_{k,m=0}^n, \qquad M_{k,m} = \int_\C z^k \bar{z}^m \, \dd \mu(z).
\]
These matrices are positive definite. To see this, note that for $\textbf{a} = (a_0, \dots, a_n) \in \C^{n+1} \setminus \{0\}$, if we put $P_n(z) = \sum_{k=0}^n a_k z^k$, then
\begin{align*}
0 & < \left\| P_n \right\|^2_{L^2(\mu)} = \int_\C \left| \sum_{k=0}^n a_k z^k \right|^2 \dd \mu(z) = \sum_{k,m=0}^n a_k \overline{a_m} \int_\C z^k \bar{z}^m \, \dd \mu(z)\\
& = \sum_{k,m=0}^n a_k \overline{a_m} M_{k,m} = \textbf{a}^\ast \cM(n;\mu) \textbf{a}.
\end{align*}

Now put $D_n = D_n(\mu) = \det \cM(n ; \mu)$. We can define \emph{monic} orthogonal polynomials with respect to $\mu$ by
\[
Q_0(z) = 1, \quad Q_n(z) = \frac{1}{D_{n-1}} \det
\left[
\begin{array}{cccc}
M_{0,0} & & & M_{0,n} \\
\vdots & \ddots & & \vdots \\
M_{n-1,0} & & & M_{n-1,n} \\
1 & z & \dots & z^n\\
\end{array}
\right].
\]
The orthogonality follows, since for $j \in \{ 0, \dots, n \}$
\[
D_{n-1} \int_\C z^j \overline{Q_n(z)} \, \dd \mu(z) =
\det
\left[
\begin{array}{cccc}
M_{0,0} & & & M_{0,n} \\
\vdots & \ddots & & \vdots \\
M_{n-1,0} & & & M_{n-1,n} \\
M_{j,0} & M_{j,1} & \dots & M_{j,n}\\
\end{array}
\right]
=
\begin{cases}
0 & ,\, j \in \{0, \dots, n-1\};\\
D_n & ,\, j = n.\\
\end{cases}
\]

From the above we also find that
\[
\left\| Q_n \right\|^2_{L^2(\mu)} = \int_\C z^n \overline{Q_n(z)} \, \dd \mu(z) = \frac{D_n}{D_{n-1}},
\]
while $D_0 = M_{0,0} = \mu(\C) = \left\| Q_0 \right\|^2_{L^2(\mu)}$. We conclude that
\[
D_n = \prod_{j=0}^n \left\| Q_j \right\|^2_{L^2(\mu)}.
\]

\section{Bergman polynomials on two disks} \label{sec:Bergman_poly_two_islands}

Here we proved Theorem \ref{thm:BegmanAsymp}. We use the method of the paper \cite{BergmanPoly} (simplified for a domain which is the union of two disks).

Let $\eta \in (0,1)$, $G = \D_1 \cup \D_2$ be the \emph{disjoint} union of two disks, whose radii are bounded between $\eta$ and $\frac{1}{\eta}$, and are separated by at least $\eta$ and at most $\frac{1}{\eta}$. Put $\Gamma = \partial G = \T_1 \cup \T_2$.

We denote by $\{ t_n \}_{n=0}^\infty$ the sequence of \emph{monic} Bergman orthogonal polynomials with respect to area measure on $G$. Write $\gamma_k$ for the normalizing constant of $t_k$ (here we do not consider the weighted case of Bergman polynomials).

For $p$ a polynomial of degree $n$ we have by the orthogonality of the monomials
\[
\int_{r \T} |p(z)|^2 \dd|z| \le \frac{2(n+1)}{r^2} \int_{r \D} |p(z)|^2 \dd A(z).
\]
Recall that $\beta_k^2$ minimizes $\int_\G |p_k|^2 \dd |z|$ for monic polynomials, and $\gamma_k^2$ minimizes $\int_G |p_k|^2 \dd A(z)$. Therefore, we have
\[
\gamma_n^2 = \int_{\D_1} |t_n(z)|^2 \dd A(z) + \int_{\D_2} |t_n(z)|^2 \dd A(z)
\ge \frac{f_3(\eta)}{n+1} \int_{\G} |t_n(z)|^2 \dd |z| \ge \frac{f_3(\eta)}{n+1} \beta_n^2,
\]
where $f_3$ is some positive continuous function in $\eta \in (0,1)$.

Using Theorem \ref{thm:Szego_OP}, we conclude the lower bound in Theorem \ref{thm:BegmanAsymp}
\[
\gamma_n \ge C_1 \cdot \frac{\logcap{\G}^{n+1}}{\sqrt{n+1}}.
\]

To prove the upper bound for $\gamma_n$ we will use the following (very special case of) \cite[Thms. $8.3$ and $5.6$]{Widom69}, concerning Chebyshev polynomials.
\begin{thm}\label{lem:Widom_poly}
Let $\G = \partial G = \T_1 \cup \T_2$ be the boundary of the disks, as in Theorem \ref{thm:BegmanAsymp}. There is a constant $M$ depending continuously on the Green function $\greenf$ of $\G$ (w.r.t. $\infty$) and its derivatives, such that for each $n$, there is a monic polynomial $Q$ of degree $n$ satisfying
\[
\log|Q(z)| \le M + n \log \logcap{\G}, \qquad \forall z\in\G.
\]
\end{thm}

\begin{rmk}
By the assumptions $M$ can be bounded (continuously) in terms of $\eta$.
\end{rmk}

Now we repeat the proof of \cite[Lemma 8.2]{BergmanPoly} in a simplified form. Let $\greenf = \greenf_\Omega$ be the Green function (always w.r.t. $\infty$) for
\[
\Omega \defeq G^c = \C \setminus (\D_1 \cup \D_2).
\]
The function $\greenf$ can be extended harmonically into the Green function (up to a constant) of a larger domain $\Omega_\tau$ which corresponds to the level set $\greenf = \log \tau$ with $\tau \in (0,1)$. This $\tau$ may be chosen to be a continuous function of $\eta$. Moreover, fixing some $\rho = \rho_{\eta} \in (\tau, 1)$ in a continuous way in $\eta$, we have for $R \ge \rho$
\[
\greenf_R(z) \defeq \greenf_{\Omega_R}(z) = \greenf(z) - \log R, \qquad \Omega_R \defeq \{ \greenf \ge \log R \}.
\]
Thus $\logcap{\partial \Omega_R} = R \logcap{\partial G} = R \logcap{\G}$. By Theorem \ref{lem:Widom_poly} applied to $\partial \Omega_\rho$
we have a polynomial $Q$ of degree $n$ such that
\[
\log |Q(z)| \le M + n \log \logcap{\partial \Omega_\rho}, \qquad \forall z \in \partial \Omega_\rho,
\]
where $M$ depends continuously on $\eta$. By the maximum principle (for the unbounded domain $\Omega_\rho$) this inequality extends to
\[
\log |Q(z)| \le M + n (\log \logcap{\partial \Omega_\rho} + \greenf_\rho(z)), \qquad \forall z \in \Omega_\rho.
\]
Since $\greenf_R(z) + \log \logcap{\partial \Omega_R} = \log \logcap{\G}$ does not depend on $R$, we have for $R \ge \rho$
\[
\log|Q(z)| \le M + n \log \logcap{\partial \Omega_R} \implies |Q(z)| \le M^\prime R^n \logcap{\Gamma}^n, \qquad \forall z \in \partial \Omega_R.
\]
Now, by the usual maximum principle
\begin{equation}\label{eq:max_principle_bnd}
|Q(z)| \le M^\prime R^n \logcap{\Gamma}^n, \qquad z \in \Omega_R^c,\: \rho \le R < \infty.    
\end{equation}

We proceed to bound the $L^2(G)$-norm of $Q$ as in \cite{BergmanPoly} by splitting $G$ into $\Omega_\rho^c$ and the rest. We bound the first part using \eqref{eq:max_principle_bnd}. The integral over $G \setminus \Omega_\rho$ is bounded by foliating it by the level sets of $\partial \Omega_R$ using \eqref{eq:max_principle_bnd} and the coarea formula. For the convenience of the reader we provide the details below.

Write the integral over $G$ as follows
\[
\int_{G} |Q|^2 \dd A = \int_{\Omega_\rho^c} |Q|^2 \dd A + \int_{G \cap \Omega_\rho} |Q|^2 \dd A \eqdef I_1 + I_2.
\]
We bound $I_1$ using \eqref{eq:max_principle_bnd} by
\[
I_1 \le \Area(G) \max_{z \in \partial \Omega_\rho} |Q|^2 \le (M^\prime)^2 \Area(G) \rho^{2n} \logcap{\Gamma}^{2n}.
\]
In order to bound $I_2$ we foliate $G \cap \Omega_\rho$ by the level sets $\partial \Omega_R$ of $\greenf_\Omega$, and use the coarea formula. Notice that $\nabla \greenf_\Omega$ is (uniformly) bounded away from zero on $G \cap \Omega_\rho$. Thus, by \eqref{eq:max_principle_bnd}
\[
I_2 = \int_{\rho}^1 \int_{\partial \Omega_R} \frac{|Q(z)|^2}{|\nabla \exp\left(\greenf_\Omega(z)\right)|} \dd |z| \dd R
\le M^{\prime\prime} \int_\rho^1 \Length(\partial \Omega_R) (M^\prime)^2 R^{2n} \logcap{\Gamma}^{2n} \dd R.
\]

In order to bound the length of the level sets (uniformly), notice that
\[
-2\pi = \int_{\partial \Omega_R} \frac{\partial \greenf}{\partial n} \dd s
= \int_{\partial \Omega_R} \nabla \greenf \cdot n \, \dd s
= -2\pi \int_{\partial \Omega_R} |\nabla \greenf| \dd s
\]
so that
\[
2\pi = \int_{\partial \Omega_R} |\nabla \greenf| \dd s \ge \min_{\partial \Omega_R}|\greenf| \cdot \Length( \partial \Omega_R) \ge C_{\eta} \Length( \partial \Omega_R).
\]
Finally this gives the following bound for $I_2$
\[
I_2 \le C_{\eta} \logcap{\Gamma}^{2n} \int_\rho^1 R^{2n} \dd R \le C_{\eta} \cdot \frac{\logcap{\Gamma}^{2n}}{n+1}.
\]

Now we conclude by noting that
\[
\gamma_n^2 = \int_G |t_n|^2 \dd A(z) \le \int_G |Q|^2 \dd A(z) \le C_{\eta}^\prime \cdot\frac{\logcap{\Gamma}^{2n+2}}{n+1}.
\]

\printbibliography

\end{document}